\documentclass[a4paper,11pt]{amsart}
\usepackage[british]{babel}
\usepackage[T1]{fontenc}
\usepackage{lmodern}
\usepackage{dirtytalk}
\usepackage[utf8]{inputenc}
\usepackage{graphicx}
\usepackage{doc}

\usepackage{amssymb}
\usepackage{amsfonts}
\usepackage{amsmath}

\usepackage{relsize,exscale}

\usepackage{tikz-cd} 

\usepackage{graphicx}
\usepackage[all]{xy}
\usepackage{array}
\usepackage{enumerate}

\usepackage{amsthm}
\usepackage{amsmath}

\newtheorem{thm}{Theorem}[section]
\newtheorem{lm}[thm]{Lemma}
\newtheorem{prop}[thm]{Proposition}
\newtheorem{cor}[thm]{Corollary}

\theoremstyle{definition}
\newtheorem{definition}[thm]{Definition}
\newtheorem{conj}[thm]{Conjecture}
\newtheorem{prop-defi}[thm]{Proposition-Definition}

\newtheorem{notation}[thm]{Notation}

\newtheorem{ex}[thm]{Example}

\theoremstyle{remark}
\newtheorem{rem}[thm]{Remark}

\newcommand{\Zdz}{\mathbb{Z}/2\mathbb{Z}}

\newcommand{\Pbold}{\boldsymbol{\mathcal{P}}}
\newcommand{\Vbold}{\boldsymbol{\mathcal{V}}}
\newcommand{\Ibold}{\mathbf{I}}
\newcommand{\Gammabold}{\boldsymbol{\Gamma} } 
\newcommand{\Sbold}{\textbf{S} } 
\newcommand{\Qbold}[1]{\ensuremath{\textbf{Q}^{#1}}}
\newcommand{\PF}{\boldsymbol{\Pi}}

\newcommand{\svec}{\mathfrak{svec}}

\newcommand{\ev}{\mathrm{ev}}
\newcommand{\Ext}{\mathrm{Ext}}
\newcommand{\Hom}{\mathrm{Hom}}

\newcommand{\Pcal}{\mathcal{P}}
\newcommand{\A}{\mathcal{A}}
\newcommand{\V}{\mathcal{V}}
\newcommand{\id}{\mathrm{id}}
\newcommand{\Tot}{\mathrm{Tot}}

\numberwithin{equation}{section}
\begin{document}

\title{Additive polynomial superfunctors and cohomology}

\author{Iacopo Giordano}

\thanks{I would like to thank my PhD supervisor Antoine Touzé for his fundamental support throughout the writing of this paper.
}

\maketitle

\begin{abstract}
	We prove a classification of additive polynomial superfunctors, which allows us to compute some extensions of a superfunctor of the form $F \circ A$ where $F$ is a classical polynomial functor and $A$ is additive. We get a formula which relates these extensions to the classical ones of $F$. A possible generalisation is conjectured at the end.
\end{abstract}



\section{Introduction}

The category $\Pcal$ of strict polynomial functors was introduced by Friedlander and Suslin in \cite{FS}, where they use it to prove the cohomological finite-generation of a finite group scheme. These functors are a powerful tool to perform explicit $\Ext$-computations for polynomial $GL_n$-representations and for modules over classical Schur algebras. Further $\Ext$-computations in $\Pcal$ were performed later by a number of authors, in particular to compute generic cohomology, see e.g. \cite{FFSS,TouzeENS, TouzeUnivSS, Chalupnik} and \cite{TouzeSurvey} for a survey.

In \cite{Axtell}, Axtell introduced a $\Zdz$-graded version of strict polynomial functors, adapted to the context of super representation theory. They are called \emph{strict polynomial superfunctors} and they have already been successfully used  by Drupieski to prove the cohomological finite-generation for finite supergroup schemes \cite{Drupieski}. However, the $\Ext$-computations made so far in the category $\Pbold$ of polynomial superfunctors are not as far-reaching as the classical ones made in $\Pcal$. 
%
%
This article is a contribution to their development.
%
Our main result is a formula which computes $\Ext$ when both arguments are constructed by composing an additive superfunctor and a classical functor.
\begin{thm}[\ref{thm-main}]
	Let $F,A,B$ be homogeneous strict polynomial functors and suppose that $A$ and $B$ are additive. Let $I^{(r)}$ denote the $r$-th Frobenius twist functor. Then there is an isomorphism, natural in all variables
	\begin{equation}\label{eqmain}
	\Ext^*_\Pcal(I^{(r)},F)\otimes \Ext^*_{\Pbold}(A,B)^{(r)} \simeq \Ext_{\Pbold}^*(I^{(r)}\circ A,F\circ B)\;.
	\end{equation}
\end{thm}

Our result relies mainly on two ingredients. In first place, we prove a classification of the additive polynomial superfunctors, showing that they are all direct sums of four fundamental types. This makes our computations easier, but more importantly it allows to actually give sense to the composition $F\circ A$ of a classical functor $F$ by an additive superfunctor $A$ (see Corollary \ref{precompaddfunctors}). It is also a result of independent interest. For example, the study of exponential functors \cite{TouzeIMRN} has shown that the classification of additive functors is a key in understanding exponential functors, and similar results are expected to hold in the super setting. 

Secondly, our $\Ext$-computations use in an essential way the ones made by Drupieski in \cite{Drupieski}. 

In the last section of the paper we rewrite Theorem 1.1 in a different way, which makes it suggestively similar to the classical formula computing general extensions $\Ext^*_{\Pcal}(G\circ I^{(r)}, F\circ I^{(r)})$, see \cite[Cor. 5]{TouzeUnivSS}, \cite[Cor. 3.7]{Chalupnik} or \cite{TouzeSurvey}.
This makes us think that the extensions of the form $\Ext^*_{\Pbold}(G\circ A, F\circ B)$ may be computed by an analogous formula, which we state in Conjecture \ref{conj}.
%
%
%
%
In order to prove the general case of the conjecture, we hope to follow the same idea as in \cite[Section 5]{TouzeSurvey}. This is to say, we aim to search some super analogues of the universal classes constructed in \cite{TVdK} and use them to produce the isomorphism of the conjecture by means of cup products, in a similar way as we have defined the isomorphism of Theorem 1.1. The construction of the universal classes made by Touzé relied on the use of Troesch $p$-complexes, introduced in \cite{Troesch}. To find these super classes one should then investigate the \say{superized} Troesch complex recently constructed by Drupieski and Kujawa \cite{DrupKujawa}.


\section{Recollections of strict polynomial (super)functors}

\subsection{Categories of strict polynomial (super)functors} We briefly recall the notations and the definition of a strict polynomial functor \cite{FS}. From now on, $k$ is a field of positive characteristic $p \ge 3$. In what follows all algebraic structures will be on $k$ (if not specified otherwise) so that we drop $k$ everywhere.

Denote by $\V$ the category of finite dimensional vector spaces and linear morphisms. We define a new category starting from $\V$. Let the symmetric group $\Sigma_n$ act on the tensor product $V^{\otimes n}$ by permutation, then $\Gamma^n V$ (resp. $S^n V$) is the vector space of the invariants (resp. coinvariants) of this action. We set $\Gamma^n \V$ as the category with the same objects as $\V$ but with morphisms
\[ \Hom_{\Gamma^n \V} (V,W) := \Gamma^n \Hom(V,W) \: . \] 
Composition is defined in a natural way via the isomorphism $\Gamma^n \Hom(V,W) \simeq  \Hom_{k \Sigma_n}(V^{\otimes n}, W^{\otimes n})$. Both $\V$ and $\Gamma^n \V$ are clearly enriched over $\V$.

\begin{definition}
	A homogeneous \emph{strict polynomial functor} of degree $n$ is a linear (covariant) functor $F : \Gamma^n \V \rightarrow \V$. The category with the latter as objects and with natural transformations between them is noted by $\mathcal{P}_n$. Finally, the category of arbitrary strict polynomial functors is $\mathcal{P} := \prod_{n \ge 0} \mathcal{P}_n$.
\end{definition}

\begin{ex}
	Let $F$ be one of the symbols $S^n, \Gamma^n, \Lambda^n$. Any $\Sigma_n$-equivariant map $V^{\otimes n} \rightarrow W^{\otimes n}$ descends to a well-defined map $F(V) \rightarrow F(W)$. Then the isomorphism mentioned above $\Gamma^n \Hom(V,W) \simeq \Hom_{k S_n}(V^{\otimes n}, W^{\otimes n})$   induces a linear morphism
	\[\Gamma^n \Hom(V,W) \rightarrow \Hom(F(V), F(W)) \]
	which makes $F$ into a strict polynomial functor of degree $n$.
\end{ex}

We now generalise this definition in a graded sense. A \emph{super vector space} is a $\Zdz$-graded vector space, namely of the form
\[V = V_0 \oplus V_1 \]
where $V_0$ is said to be its \emph{even} part, while $V_1$ the \emph{odd} part. If $v \in V$ is a homogeneous vector, we denote by $\overline{v} \in \Zdz$ its super-degree. The \emph{superdimension} of $V$ is $sdim(V) := (dim V_0, \: dim V_1)$. Any super vector space of superdimension $(m,n)$ identifies with $k^{m|n} := k^m \oplus k^n$. A linear map between two super vector spaces $f:V\rightarrow W$ is said \emph{even} (resp. \emph{odd}) if it preserves (resp. swaps) the even and odd parts of $V$ and $W$. Every linear map decomposes uniquely as the sum of an even and an odd map.

Let $\mathfrak{svec}$ be the category of super vector spaces and linear morphisms, and let $\Vbold$ be the full subcategory of finite dimensional objects. The last remark implies that $\mathfrak{svec}$ and $\Vbold$ are both enriched over themselves (or \emph{superlinear}). 

Given $V \in \Vbold$, there is a right action of $\Sigma_n$ on $V^{\otimes n}$ determined by its formula on transpositions $(i, i+1)$:
\[ (v_1 \otimes \dots \otimes v_n) \cdot (i, i+1) := (-1)^{\overline{v_i} \:\overline{v_{i+1}}}(v_1 \otimes \dots \otimes v_{i+1} \otimes v_i \otimes \dots \otimes v_n) \]
Set $\Gammabold^n V$ and $\Sbold^n V$ to be respectively the invariants and coinvariants of $V^{\otimes n}$ under this action. The $\Zdz$-graduation on them is induced by the isomorphisms 
\begin{align*}
\Sbold^n(V) \simeq \bigoplus_{a+b=n} S^a(V_0) \otimes \Lambda^b(V_1) \\
\Gammabold^n(V) \simeq \bigoplus_{a+b=n} \Gamma^a(V_0) \otimes \Lambda^b(V_1)
\end{align*}
where in both formulas $V_i$ is placed in degree $i$.
We define the category $\Gammabold^n \Vbold$ which has the same objects as $\Vbold$ and morphisms
\[ \Hom_{\Gammabold^n \Vbold} (V,W) := \Gammabold^n \Hom(V,W) \: .\] 
This implies that it is a superlinear category too.

\begin{definition}
	A homogeneous \emph{strict polynomial superfunctor} of degree $n$ is an even superlinear functor $F: \Gammabold^n \Vbold \rightarrow \Vbold$, i.e. a covariant functor such that the map $\Gammabold^n \Hom(V,W) \rightarrow \Hom(F(V), F(W))$ is linear and even for any $V,W \in \Vbold$.
\end{definition}

\begin{ex}
	$\Sbold^n$ and $\Gammabold^n$ are strict polynomial superfunctors of degree $n$. 
\end{ex}

There is an important example of degree $1$, other than the identity $\Ibold$. For each $V$, let $\PF(V)$ be the super vector space with exchanged parities, i.e. having $V_1$ as even part and $V_0$ as odd part. If $\varphi:V \rightarrow W$ is a linear map, let $\PF(\varphi) = (-1)^{\overline{\varphi}} \varphi : \PF(V) \rightarrow \PF(W)$. Then $\PF$ defines a strict polynomial superfunctor of degree $1$. 

If $F,G$ are polynomial of degree $n$, a \emph{natural transformation} $f: F \rightarrow G$ is a collection of linear maps $\{f_V : F(V) \rightarrow G(V), \; V \in \Vbold\}$ such that for any $\varphi \in \Gammabold^n \Hom(V,W)$ 
\[ f_W \circ F(\varphi) = (-1)^{\overline{\varphi} \: \overline{f}} G(\varphi) \circ f_V \]
\begin{definition}\label{Pboldcategory}
	For each $n\geq0$, $\Pbold_n$ denotes the category of homogeneous polynomial superfunctors of degree $n$ and natural transformations. We set $\Pbold := \prod_{n \ge 0} \Pbold_n$ to be the category of all polynomial superfunctors.
\end{definition}

We will always consider only homogeneous polynomial (super)functors, since all the assertions on non-homogeneous ones can be recovered from them by Definition \ref{Pboldcategory}. In particular, if $F,G$ are homogeneous, $\Hom_{\Pbold}(F,G)$ can be nonzero only if they have the same degree $d$, in which case $\Hom_{\Pbold}(F,G)$ is equal to $\Hom_{\Pbold_d}(F,G)$. We will then often write $\Hom_{\Pbold}(-,-)$ instead of $\Hom_{\Pbold_d}(-,-)$ to ease notation. The same things can be stated for $\mathcal{P}$. Moreover, when there is no possibility of confusion with other kinds of functors, we will use the bare term \emph{superfunctor} (resp. \emph{functor}) to denote an element of $\Pbold$ (resp. $\mathcal{P}$).

Let $F \in \Pbold_n, G \in \Pbold_m$. The tensor product $F \otimes G$, defined by $V \mapsto F(V)\otimes G(V)$, is polynomial of degree $n+m$. The composed functor $F \circ G$, of degree $nm$, is defined on objects by $(F \circ G)(V) = F(G(V))$ and on morphisms by
\[
\begin{tikzcd}
\Gammabold^{nm}\Hom(V,W) \arrow[r] & \Gammabold^n(\Gammabold^m \Hom(V,W)) \arrow[r, "\Gammabold^n G"] & \Gammabold^n \Hom(G(V),G(W)) \arrow[d, "F"] \\
& & \Hom(F(G(V)), F(G(W)))  
\end{tikzcd}
\]
where the first map is induced by the inclusion $\Sigma_m^{\times n} \subset \Sigma_{nm}$. We give a last example of standard construction: 

\begin{definition}
	Let $F \in \Pbold_n$. The \emph{Kuhn dual} of $F$ is the functor $F^\# \in \Pbold_n$ defined by $F^\#(V) := (F(V^*))^*$.
\end{definition}

\begin{rem}\label{kuhnduals}
	For example, $\Ibold$ and $\PF$ are self-dual, while $\Sbold^n$ and $\Gammabold^n$ are the dual of each other.
\end{rem}

\begin{rem}\label{kuhnduality}
	Kuhn duality induces an isomorphism, natural in either variable
	\[ \Hom_{\Pbold}(F,G) \simeq \Hom_{\Pbold}(G^\#, F^\#) \]
\end{rem}

\subsection{Homological algebra in $\Pbold$}
The category $\Pbold$ is not abelian\footnote{The reason is that there is no natural $\Zdz$-grading on the (co)kernel of a non-homogeneous map.}, so we must specify what we mean by taking extensions in $\Pbold$.

For any superlinear category $\mathcal{C}$, let $\mathcal{C}_{ev}$ denote the category with the same objects as $\mathcal{C}$ but only even morphisms (in particular the $\Hom_{\mathcal{C}_{ev}}(-,-)$ are purely even spaces). 
Then $\Pbold_{ev}$ is an abelian category, where kernels and cokernels are computed espace-wise. Now, the parity change functor $\PF$ defined in the previous section induces an isomorphism of super vector spaces, natural in $F,G$ with respect to even transformations
\begin{equation}\label{homparitychange}
\Hom_{\Pbold}(F,G) \simeq \Hom_{\Pbold_{ev}}(F,G) \oplus  \Hom_{\Pbold_{ev}}(\PF \circ F,G)
\end{equation}
where the left (right) summand is placed in even (odd) degree.
The starting point for cohomological computations is the Yoneda lemma. Given $V \in \Vbold$, consider the superfunctor $\Gammabold^{d,V} := \Gammabold^d \: \Hom(V,-)$.

\begin{thm}\label{yoneda}
For any $F \in \Pbold_d$ there is an isomorphism, natural in $V$ and in $F$:
\[ \Hom_{\Pbold_d}(\Gammabold^{d,V}, F) \simeq F(V)  \]
\end{thm}
In particular, by (\ref{homparitychange}), $\Gammabold^{d,V} \oplus (\PF \circ \Gammabold^{d,V})$ is a projective object in $(\Pbold_d)_{ev}$. In general, Theorem \ref{yoneda} and (\ref{homparitychange}) imply that a natural transformation
\[ T: \Gammabold^{d,V} \oplus \PF \circ \Gammabold^{d,V} \rightarrow F \]
in $\Pbold_{ev}$ is completely determined by a vector $v = v_0 + v_1 \in F(V)$, explicitly by the formula $T(f + \pi g) = Ff(v_0) + Fg(v_1)$. Moreover, the \emph{Yoneda morphism} (of which we keep track for later use):
\begin{align}\label{explicityonedabismorphism}
\begin{split}
(\Gammabold^{d,V} \oplus \PF \circ \Gammabold^{d,V}) \otimes F(V) \longrightarrow F \\
(f + \pi g) \otimes v \longmapsto Ff(v_0) + Fg(v_1)
\end{split}
\end{align}
induces an epimorphism
\[ \bigoplus_{m,n \ge 1} (\Gammabold^{d,k^{m|n}} \oplus \PF \circ \Gammabold^{d,k^{m|n}}) \otimes F(k^{m|n}) \longrightarrow F  \]
so that the $\Gammabold^{d,V} \oplus (\PF \circ \Gammabold^{d,V})$ form a set of projective generators of $({\Pbold_d})_{ev}$ when $V$ runs through $\Vbold$ \cite[Thm. 3.1.1]{Drupieski}. Thus, $\Pbold_{ev}$ is an abelian category with enough projectives and one can compute extensions in it. Extending (\ref{homparitychange}), we define the \say{extensions in $\Pbold$} as
\[ \Ext^*_{\Pbold}(F,G) \simeq \Ext^*_{\Pbold_{ev}}(F,G) \oplus  \Ext^*_{\Pbold_{ev}}(\PF \circ F,G) \]
with the left (right) summand placed in even (odd) degree.



\section{Classification of additive superfunctors}

\subsection{Additive superfunctors}

Let $F$ be a polynomial superfunctor. Given super vector spaces $V,W$, the inclusions $V, W \subset V\oplus W$ induce by funtoriality inclusions $F(V), F(W) \subset F(V\oplus W)$.  

\begin{definition}
	A superfunctor $F$ is \emph{additive} if the inclusions induce an isomorphism $F(V\oplus W) \simeq F(V) \oplus F(W)$.
\end{definition} 

Two basic examples are $\Ibold$ and $\PF$. Other fundamental examples are provided by Frobenius twists. To be more specific, for a super vector space $V$ and a positive integer $r$, set $V^{(r)} := k \:\otimes_{\varphi} V$, where $\varphi : k \rightarrow k$ is the $p^r$-th power map. If $V=V_0 \oplus V_1$ is the decomposition of $V$ in even and odd part, set $V_0^{(r)}$ to be the even part and $V_1^{(r)}$ to be the odd part of $V^{(r)}$. 

\begin{definition}
	The super vector space $V^{(r)}$ is called the $r$-th \emph{Frobenius twist} of $V$. 
	We use the notation $v^{(r)} := 1 \otimes v$ to indicate its elements. 
\end{definition}

Set $\Ibold^{(r)}(V) = V^{(r)}$. We want to make it into a polynomial superfunctor. To do that,  consider the linear inclusion $i: V^{(r)} \hookrightarrow \Sbold^{p^r}(V), \; \; v^{(r)} \mapsto v^{p^r}$ which induces by duality a surjection $i^\# : \Gammabold^{p^r}(V) \rightarrow V^{(r)}$. In fact, if $v=~v_0~+v_1$ is the decomposition of $v$ in even and odd components, then $v^{p^r} = (v_0)^{p^r}$, so that the image of $i$ is contained in the subspace $\Sbold^{p^r}(V_0)$. Dually, this means that $i^\#$ vanishes anywhere but on $\Gammabold^{p^r}(V_0)$, in particular its image is contained in $V_0^{(r)}$. 
We make use of this observation to define $\Ibold^{(r)}$ on morphisms. If $f:V\rightarrow W$ is a homogeneous linear morphism, then $f^{(r)}: V^{(r)} \longrightarrow W^{(r)}$ defined by $f^{(r)}(v^{(r)}):= (f(v))^{(r)}$ is linear of the same parity, and any linear morphism $V^{(r)} \rightarrow W^{(r)}$ is built like that. This defines an even isomorphism $\Hom(V^{(r)}, W^{(r)}) \simeq (\Hom(V,W))^{(r)}$. Composing it with $i^\#$ gives an even linear map
\begin{equation}\label{aaa}
 \Gammabold^{p^r}(Hom(V,W)) \xrightarrow{i^\#} (\Hom(V,W)_0)^{(r)} \simeq \Hom(V^{(r)}, W^{(r)})_0  
\end{equation}
which endows $\Ibold^{(r)}$ with the structure of a polynomial superfunctor of degree $p^r$. Moreover, the fact that (\ref{aaa}) has image in the space of even morphisms allows to define subfunctors
\[
\begin{array}{cccc}
\Ibold_0^{(r)}(V) = V_0^{(r)}, & & & \Ibold_1^{(r)}(V) = V_1^{(r)},
\end{array}
\]
such that $\Ibold^{(r)} = \Ibold_0^{(r)} \oplus \Ibold_1^{(r)}$. The following lemma is an easy verification.

\begin{lm}\label{rem}
	$\Ibold_0^{(r)}, \Ibold_1^{(r)}$ and $\Ibold^{(r)}$ are additive. Moreover, there are isomorphisms
	\[ 
	\begin{array}{cccc}
	\PF \circ \Ibold_0^{(r)} \simeq \Ibold_1^{(r)} \circ \PF \: , & & &
	\PF \circ \Ibold_1^{(r)} \simeq \Ibold_0^{(r)} \circ \PF \:.
	\end{array}
	\]
\end{lm}

Additive superfunctors have an important homological property, often referred to as Pirashvili's vanishing lemma, since it was first proved \cite{Pirashvili} by Pirashvili in the context of non-strict functors.

\begin{lm}\emph{\cite{Drupieski,FS}}\label{pirashvili}
	Let $\mathcal{A} := \Pcal$ or $\Pbold$.
	Let $A,F,G \in \mathcal{A}$ be such that $A$ is additive and $F,G$ are reduced, i.e. $F(0)=0=G(0)$. Then 
	\[\Ext_{\mathcal{A}}^*(A,F\otimes G) = \Ext_{\mathcal{A}}^*(F\otimes G, A) = 0  \]
\end{lm} %

\subsection{The classification and some consequences}\label{sec-31}

The next theorem states that the examples of additive superfunctors given in Section 3.1 are essentially the only ones:

\begin{thm}\label{classaddfunc}
Let $A$ be an additive homogeneous superfunctor of degree $d$. Then there is an $r\ge 0$ such that $d=p^r$. Moreover:
\begin{enumerate}
\item if $r=0$ then $A$ is a direct sum of copies of the functors $\Ibold$ and $ \PF$,
\item if $r>0$ then $A$ is a direct sum of copies of the functors $\Ibold^{(r)}_0$, $ \Ibold^{(r)}_1$, $\Ibold^{(r)}_0\circ \PF$ and $\Ibold^{(r)}_1\circ \PF$.
\end{enumerate}
\end{thm}

\begin{rem}
$\Ibold, \PF, \Ibold^{(r)}_0$, $ \Ibold^{(r)}_1$, $\Ibold^{(r)}_0\circ \PF, \Ibold^{(r)}_1\circ \PF$ are indecomposable additive superfunctors. Theorem \ref{classaddfunc} shows that they are the only ones.
\end{rem} 

\begin{rem}
	The decomposition of a functor $F \in \Pbold_{p^r}$ given by Theorem \ref{classaddfunc} can be written in a more precise way, even if less concise. In fact, for $r>0$ the number of copies of $\Ibold_0^{(r)}$ (resp. $\Ibold^{(r)}_0\circ \PF, \; \Ibold^{(r)}_1\circ \PF, \; \Ibold^{(r)}_1$) is equal to the dimension of the vector space $F(k)_0$ (resp. $F(k)_1, F(\Pi k)_0, F(\Pi k)_1$, where $\Pi k := k^{0|1}$). For $r=0$ the number of copies of $\Ibold$ (resp. $\PF$) is equal to the dimension of $F(k)_0$ (resp. $F(k)_1$). This ensures in particular the unicity of the decomposition.
\end{rem}

An important consequence of Theorem \ref{classaddfunc} is that we can precompose a polynomial functor by an additive superfunctor to obtain a new superfunctor.
To be more specific, consider the category $\mathcal{V}^*$ of finite dimensional $\mathbb{Z}$-graded vector spaces and morphisms which preserve the gradings. If $V^* \in \mathcal{V}^*$ and $F \in \Pcal_d$, one can define $F(V^*)$ in this manner \cite[Sec. 2.5]{TouzeENS}:

\begin{itemize}
	\item $F(V^*) := F(V)$ as an object, where $V$ is the underlying ungraded vector space;
	
	\item Consider the action of the multiplicative group $\mathbb{G}_m$ on $V^*$ having weight $i$ on each $V^i$. By functoriality, this induces an action of $\mathbb{G}_m$ on $F(V)$ and then a weight decomposition $F(V) = \bigoplus_{j\ge1}F(V)^j$. We take this one to be the graduation on $F(V^*)$
\end{itemize}
If $f:V^* \rightarrow W^*$ is a morphism in $\mathcal{V}^*$, then $Ff$ sends each $F(V^*)^j$ on $F(W^*)^j$ so it is a morphism in $\V^*$ too. In this way $F$ becomes a linear functor $\Gamma^d \mathcal{V}^* \rightarrow \mathcal{V}^*$.

Note that there is an inclusion of categories $\Vbold_{ev} \subset \mathcal{V}^*$, as well as a functor $\mathcal{V}^* \rightarrow \Vbold_{ev}$ reducing gradings modulo $2$. 
In particular we can restrict this construction to make $F$ into a linear functor $\Gammabold^d \Vbold_{ev} \rightarrow \Vbold_{ev}$. 

Now, let $A$ be a homogeneous additive superfunctor of degree $d>1$ with finite-dimensional values. By Theorem \ref{classaddfunc}, $d=p^r$ for some $r>0$ and $A$ can be written as a direct sum of copies of $\Ibold_0^{(r)}, \Ibold_1^{(r)}$ and their precomposition by $\PF$. From this and (\ref{aaa}) we deduce that the image of $A$ is contained in $\Vbold_{ev}$. Hence, the following composition makes sense:
\[ F \circ A : \Gammabold^{p^r d} \Vbold \xrightarrow{\Gammabold^{p^r}A} \Gammabold^d \Vbold_{ev} \xrightarrow{F} \Vbold_{ev} \subset \Vbold \]
Moreover, the induced map on the morphisms
 \[\Gammabold^{p^r d} \Hom(V,W) \rightarrow \Hom(F(A(V)), F(A(W)))\]
is even linear because it is a composition of even linear maps. Summing up, we have:

\begin{cor}\label{precompaddfunctors}
Let $A$ be an additive superfunctor of degree $p^r, \; r>0$. Then $A$ restricts to a superlinear functor 
\[\Gammabold^{p^r}\Vbold\to \svec_\ev .\]
In particular, if $A$ has finite dimensional values, precomposition by $A$ yields an exact functor
\[-\circ A:\Pcal_d\to (\Pbold_{p^r d})_{ev}\;.\]
\end{cor} 

\subsection{Proof of Theorem \ref{classaddfunc}}
We follow the approach used in \cite{TouzeAdd} to classify the additive functors in $\mathcal{P}$.
We start by stating a very useful property of additive superfunctors. For the sake of legibility, we let $k$ stand for $k^{1|0}$ and $\Pi k$ for~$k^{0|1}$ for the rest of this section.

%
%

\begin{lm}\label{naturalsurjaddfunctors}
	Let $F,G \in \Pbold$ be superfunctors and suppose that $F$ is additive (resp. that $G$ is additive). Let $\varphi: F \rightarrow G$ be a natural transformation between them. Then $\varphi$ is a monomorphism (resp. epimorphism) if and only if $\varphi_k : F(k) \rightarrow G(k)$ and $\varphi_{\Pi k} : F(\Pi k) \rightarrow G(\Pi k)$ are monomorphisms (resp. epimorphisms).
\end{lm}
\begin{proof}
	One implication is obvious. For the other, we only treat the injection case, the other being dual. We need to check that $\varphi_V : F(V) \rightarrow G(V)$ is an injection for any $V \in \Vbold$. Since $\varphi$ is natural, we may choose a homogenous basis to identify $V \simeq k^{m|n}$. Then, using the additivity of $F$, we can factor $\varphi_V$ by
	\[
	F(V) \simeq F(k^{m|n})	\simeq F(k)^{\oplus m} \oplus F(\Pi k)^{\oplus n} \xrightarrow{\simeq} G(k)^{\oplus m} \oplus G(\Pi k)^{\oplus n} \subset G(V)
	\]
	where the middle arrow is a sum of $\varphi_k$ and $\varphi_{\Pi k}$, by hypothesis isomorphisms, and the final inclusion is given by functoriality on the inclusions $k, \Pi k \subset~V$. The assertion follows. 
\end{proof}

\begin{cor}\label{naturalisoaddfunctors}
	Let $F,G \in \Pbold$ be additive and let $\varphi : F \rightarrow G$ be a natural transformation between them. Then $\varphi$ is an isomorphism if and only if $\varphi_k: F(k) \rightarrow~G(k)$ and $\varphi_{\Pi k} : F(\Pi k) \rightarrow G(\Pi k)$ are.
\end{cor}

Set now $d \geq 1$. Let $\Qbold{d}$ denote the superfunctor defined as the cokernel of the natural map induced by multiplication in $\Gammabold^d$ 
\begin{equation}\label{multingamma}
\bigoplus_{k=1}^{d-1} \Gammabold^k \otimes \Gammabold^{d-k} \longrightarrow \Gammabold^d
\end{equation}
Note that $\Qbold{1} = \Ibold$, since in that case the direct sum is equal to zero.  

\begin{lm}\label{dimensionQk}
	
$\Qbold{d}$ is an additive superfunctor. Moreover,
	\begin{enumerate}
		\item $\Qbold{d}=0$ if $d$ is not a power of $p$.
		\item $\Qbold{p^r}(k) \simeq k$ for all $r\geq 0$.
		\item $\Qbold{d}(\Pi k)\simeq \Pi k$ if $d=1$ and it is zero otherwise.
	\end{enumerate}
\end{lm}
\begin{proof}
	The proof of the first three statements is a rewriting word-for-word of \cite[Lemma 3.1 and 3.3]{TouzeAdd}. The point (3) follows from the fact that $\Qbold{1} = \Ibold$ and $\Gammabold^d(\Pi k)=\Lambda^d(k)=0$ if $d>1$. 
\end{proof}
Let $F$ be a superfunctor of degree $d$. Note that $\Gammabold^{d, k} = \Gammabold^d$ and $\Gammabold^{d, \Pi k} \simeq \Gammabold^d \circ \PF$, so that the Yoneda maps (\ref{explicityonedabismorphism}) yield morphisms
%
%
\begin{gather}\label{yonedamorphappoggio}
\begin{split}
(\Gammabold^d \oplus \PF \circ \Gammabold^d) \otimes F(k) \longrightarrow F\: , \\
(\Gammabold^d \circ \PF \oplus \PF \circ \Gammabold^d \circ \PF) \otimes F(\Pi k) \longrightarrow F\: ,
\end{split}
\end{gather} 
%
which can be decomposed into four morphisms 
\begin{gather}\label{naturalmaps}
\begin{split}
\Gammabold^d \otimes F(k)_0 \longrightarrow F\: , \\
\PF \circ \Gammabold^d \otimes F(k)_1 \longrightarrow F\: , \\
(\Gammabold^d \circ \PF) \otimes F(\Pi k)_0 \longrightarrow F \: , \\ 
(\PF \circ \Gammabold^d \circ \PF) \otimes F(\Pi k)_1 \longrightarrow F \: .
\end{split}
\end{gather}
Now evaluate the first two ones on $k$ and the second two ones on $\Pi k$. By the explicit formula in (\ref{explicityonedabismorphism}), the direct factor $\Gammabold^d (V) \otimes F(k)_0$ is mapped onto the even part $F(k)_0$ of $F(k)$, and similarly for the other ones:
\begin{gather}\label{isomorphismappoggio}
\begin{split}
\Gammabold^d(k) \otimes F(k)_0 \longrightarrow F(k)_0 \: ,\\
\PF \circ \Gammabold^d(k) \otimes F(k)_1 \longrightarrow F(k)_1 \: ,\\
\Gammabold^d(k) \otimes F(\Pi k)_0 \longrightarrow F(\Pi k)_0 \: ,\\ 
\PF \circ \Gammabold^d(k) \otimes F(\Pi k)_1 \longrightarrow F(\Pi k)_1 \: .
\end{split}
\end{gather}
Since $\Gammabold^d(k) \simeq \Gammabold^d \Hom(k,k)$ is spanned by the identity, the morphisms (\ref{isomorphismappoggio}) are immediately isomorphisms.

Let us go back for a moment to the natural maps (\ref{naturalmaps}). Precompose them by the multiplication map (\ref{multingamma}), duly composed with $\PF$ where needed. The composite map is zero by Pirashvili's vanishing lemma \ref{pirashvili}, thus the morphisms (\ref{naturalmaps}) factor through natural maps
\begin{gather}\label{Qisomorphism}
\begin{split}
\Qbold{d} \otimes F(k)_0 \longrightarrow F\: , \\
\PF \circ \Qbold{d} \otimes F(k)_1 \longrightarrow F\: , \\
(\Qbold{d} \circ \PF) \otimes F(\Pi k)_0 \longrightarrow F \: ,\\
(\PF \circ \Qbold{d} \circ \PF) \otimes F(\Pi k)_1 \longrightarrow F\: .
\end{split}
\end{gather}
As before, evaluate the first two ones on $k$ and the second two ones on $\Pi k$ to get
\begin{gather}\label{Qisomorphismappoggio}
\begin{split}
\Qbold{d}(k) \otimes F(k)_0 \longrightarrow F(k)_0 \; ,\\
\PF \circ \Qbold{d}(k) \otimes F(k)_1 \longrightarrow F(k)_1 \; ,\\
\Qbold{d}(k) \otimes F(\Pi k)_0 \longrightarrow F(\Pi k)_0 \; ,\\
\PF \circ \Qbold{d}(k) \otimes F(\Pi k)_1 \longrightarrow F(\Pi k)_1 \; .
\end{split}
\end{gather}
The morphisms (\ref{Qisomorphismappoggio}) are isomorphisms because they are the factorisation of the isomorphisms (\ref{isomorphismappoggio}) through the quotient map $\Gammabold^d(k) \rightarrow \Qbold{d}(k)$.

\begin{prop}\label{Quasistructureaddfunc}
	Let $F$ be an additive superfunctor of degree $d>1$. Then there is a natural isomorphism 
	
	\begin{center}$
		\begin{array}{l}
		F \simeq \Qbold{d} \otimes F(k)_0 \; \oplus \; (\PF \circ \Qbold{d}) \otimes F(k)_1 \\
		\; \oplus \; (\Qbold{d} \circ \PF) \otimes F(\Pi k)_0 \; \oplus \; (\PF \circ \Qbold{d} \circ \PF) \otimes F(\Pi k)_1 \: .
		\end{array}
		$\end{center}
\end{prop} 
\begin{proof}
	The candidate isomorphism is the sum of the four natural morphisms (\ref{Qisomorphism}). Lemma \ref{dimensionQk} and isomorphisms (\ref{Qisomorphismappoggio}) imply that it is an isomorphism when evaluated on $k$ and $\Pi k$, thus the assertion follows from Corollary \ref{naturalisoaddfunctors}.
\end{proof}

The analogue statement of Proposition \ref{Quasistructureaddfunc} for functors of degree $1$ requires a slight modification. Namely, we need the following:

\begin{lm}\label{app}
	Let $F$ be a strict polynomial superfunctor of degree $1$. Then there exists an odd isomorphism $F(k) \simeq F(\Pi k)$.
\end{lm}

\begin{proof}
	By hypothesis $F$ is defined on morphisms by an even linear map
	\[ F_{V,W} : \Hom(V,W) \longrightarrow \Hom(F(V), F(W)) \: . \]
	Let $\pi : k \rightarrow \Pi k$, $\pi' : \Pi k \rightarrow k$ be the parity change maps. Clearly $\pi' \circ \pi = Id_k$ and $\pi \circ \pi' = Id_{\Pi k}$. Then by functoriality $F(\pi) : F(k) \rightarrow F(\Pi k)$ is an odd isomorphism with inverse $F(\pi')$. 
\end{proof}

\begin{prop}\label{Partialclass1}
	Let $F$ be an additive superfunctor of degree $1$. Then there is a natural isomorphism
	
	\[ F \simeq \Ibold \otimes F(k)_0 \: \oplus \: \PF \otimes F(k)_1 \: . \]
\end{prop}
\begin{proof}
	The natural morphism is given by the sum of the first two ones in (\ref{Qisomorphism}) (recall that $\Qbold{1}=\Ibold$). It is clearly an isomorphism when evaluated on $k$, and it is also when evaluated on $\Pi k$ by Lemma \ref{app}. Thus it is an isomorphism by Corollary \ref{naturalisoaddfunctors}.
\end{proof}

All that is left is to identify $\Qbold{d}$.

\begin{prop}\label{QegalI0}
	Let $d\ge0$. Then $\Qbold{d} \simeq \Ibold_0^{(r)}$ if $d=p^r$ and zero otherwise.
\end{prop}
\begin{proof}
	If $d$ is not a power of $p$, Lemma \ref{dimensionQk}(1) ensures us that $\Qbold{d}=0$. Otherwise, use the fact that $\Ibold_0^{(r)}$ vanishes on purely odd spaces and sends $k$ onto itself (up to isomorphism). The statement follows then from an application of Proposition \ref{Quasistructureaddfunc} to $F=\Ibold_0^{(r)}$.
\end{proof}

We can finally gather the results of the section to prove the classification theorem.

\begin{proof}[Proof of Theorem \ref{classaddfunc}]
Propositions \ref{Quasistructureaddfunc} and \ref{QegalI0} imply in particular that a nonzero additive superfunctor has degree $p^r$. If $r=0$, it is isomorphic to a direct sum of copies of $\Ibold$ and $\PF$ by Proposition \ref{Partialclass1}; if $r>0$, it is isomorphic to a direct sum of copies of
$\Ibold_0^{(r)}, \; \PF \circ \Ibold_0^{(r)}, \; \Ibold_0^{(r)} \circ \PF$ and $\PF \circ \Ibold_0^{(r)} \circ \PF$ by Proposition \ref{Quasistructureaddfunc}. The statement follows then by the isomorphisms $\PF \circ \Ibold_1^{(r)}\simeq \Ibold_0^{(r)} \circ \PF$ and $\PF \circ \PF = \Ibold$.
\end{proof}

\section{Ext-computations}\label{sec-ext-comp}

In this section, we give a general formula to compute the extensions of the form 
\[\Ext^*_{\Pbold}(I^{(r)}\circ A,F\circ B)\]
where $I^{(r)}$ is the Frobenius twist in $\mathcal{P}$ \cite[Section 1]{FS}, $A$ and $B$ are additive homogeneous polynomial superfunctors of the same degree, strictly greater than $1$, and $F$ is a polynomial functor of degree $p^r$. 

We will make constant use of the Yoneda composition of extensions. We denote by $y\cdot x$ the Yoneda composition of $x\in \Ext^*_\A(X,Y)$ and $y\in\Ext^*_\A(Y,Z)$ in an abelian category $\A$. Then $\Ext^*_\A(X,-)$ is a functor from $\A$ to the category of right modules over the Yoneda algebra $\Ext^*_\A(X,X)$.

Throughout the rest of the section we fix $\ell\in\{0,1\}$. When an $\overline{\ell}$ appears nearby, it means $1-\ell$.

\subsection{The case $A=\Ibold^{(s)}_\ell=B$}\label{secsigma}

We first recall from \cite{Drupieski} the following fundamental computation. Here $(S^{p^s})_\ell^{(t-s)} := S^{p^s} \circ \Ibold_\ell^{(t-s)}$.
\begin{prop}\emph{\cite[Thm 4.5.1, Cor. 4.6.2 and 4.6.5]{Drupieski}\label{pr-alg}~} 
\begin{enumerate}[(1)]
\item 
The Yoneda algebra $\Ext_{\Pbold}^*(\Ibold_\ell^{(t)}, \Ibold_\ell^{(t)})$ is a graded commutative algebra generated by elements $\{e_t^1, \dots, e_t^{t}\}$  with each $e_t^i$ of degree $2p^{i-1}$, and with relations $(e_t^1)^p = \dots = (e_t^{t-1})^p = 0$. 

\item $\Ext_{\Pbold}^*(\Ibold_\ell^{(t)}, \Ibold_{\overline{\ell}}^{(t)})$ is a graded free $\Ext_{\Pbold}^*(\Ibold_\ell^{(t)}, \Ibold_\ell^{(t)})$-module generated by an element $c_t$ of degree $p^t$.

\item For $0\le s<t$, the canonical morphism $\Ibold_\ell^{(t)}\to (S^{p^s})_\ell^{(t-s)}$ induces surjective maps:
\[\Ext_{\Pbold}^*(\Ibold_\ell^{(t)}, \Ibold_\ell^{(t)})\to \Ext_{\Pbold}^*(\Ibold_\ell^{(t)}, (S^{p^s})_\ell^{(t-s)})\;,\]

\[\Ext_{\Pbold}^*(\Ibold_\ell^{(t)}, \Ibold_{\overline{\ell}}^{(t)})\to \Ext_{\Pbold}^*(\Ibold_\ell^{(t)}, (S^{p^s})_{\overline{\ell}}^{(t-s)})\;,\]
whose kernels are the right $\Ext_{\Pbold}^*(\Ibold_\ell^{(t)}, \Ibold_\ell^{(t)})$-modules respectively generated by $\{e_t^1,\dots,e_t^s\}$ and $\{e_t^1 \cdot c_t,\dots, e_t^s \cdot c_t.\}$
\end{enumerate}
\end{prop}

\begin{rem}
	The canonical morphism $\Ibold_0^{(t)}\to (S^{p^s})_0^{(t-s)}$ mentioned in the Proposition is obtained by precomposing with $\Ibold_\ell^{(t-s)}$ the $s$-th power map $I^{(s)}\to S^{p^s}$. In other words, it is defined by $v^{(t)} \mapsto (v^{p^s})^{(t-s)}$. For the same morphism with $\ell = 1$, use the isomorphism $(S^{p^s})_1^{(t-s)} \simeq \PF \circ  (S^{p^s})_0^{(t-s)} \circ \PF$ (which in particular for $s=0$ gives $\Ibold_1^{(t)} \simeq \PF \circ \Ibold_0^{(t)} \circ \PF$).
\end{rem}

The following notation agrees with the procedure described in section \ref{sec-31}, evaluating a strict polynomial functor $F$ on a graded vector space.
\begin{notation}\label{nota-twist-ev}
Given a graded vector space $V$, we denote by $V^{(r)}$ the graded vector space with
\[(V^{(r)})^d=
\begin{cases}
(V^i)^{(r)} &\text{if $d=p^ri$}\\
0 & \text{else}
\end{cases}.
\]
\end{notation}

The following result is a straightforward consequence of Proposition \ref{pr-alg} and notation \ref{nota-twist-ev}.

\begin{cor}
There is an injective morphism of graded algebras
\[\sigma:\Ext_{\Pbold}^*(\Ibold^{(s)}_\ell,\Ibold^{(s)}_\ell)^{(r)}\to \Ext_{\Pbold}^*(\Ibold^{(r+s)}_\ell,\Ibold^{(r+s)}_\ell)\]
given by $\sigma((e_s^i)^{(r)})=e_{s+r}^{i+r}$ for $1\le i\le s$.
\end{cor}


Let now $F \in \Pcal_{p^r}$. We can define a graded linear map
\[\begin{array}{cccc}
\Psi:& \Ext^*_\Pcal(I^{(r)},F)\otimes \Ext^*_{\Pbold}(\Ibold^{(s)}_\ell,\Ibold^{(s)}_\ell)^{(r)}&\to &\Ext_{\Pbold}(I^{(r)}\circ \Ibold^{(s)}_\ell,F\circ \Ibold^{(s)}_\ell)\;,\\
& x\otimes e^{(r)} & \mapsto & (x\circ \Ibold^{(s)}_\ell)\cdot \sigma(e^{(r)})
\end{array}
\]

\begin{lm}\label{lm-premier-cas}
If  $F$ is injective, then $\Psi$ is an isomorphism.
\end{lm}
\begin{proof}
The symmetric tensors $S^\mu := S^{\mu_1} \otimes \dots \otimes S^{\mu_n}$, for all partitions $\mu = (\mu_1, \dots, \mu_n)$ of $p^s$ and all $n\ge1$, form a set of injective cogenerators of $\Pcal_{p^s}$, so it suffices to check that $\Psi$ is an isomorphism if $F=S^\mu$.
If $\mu=(p^s)$, then $\Ext_{\Pcal}^*(I^{(r)},S^{p^r})$ is zero in positive degrees and one-dimensional in degree $0$. So checking that $\Psi$ is an isomorphism reduces to checking that the map 
\[\Ext_{\Pbold}^*(\Ibold_\ell^{(s)}, \Ibold_\ell^{(s)})^{(r)}\to \Ext_{\Pbold}^*(\Ibold_\ell^{(r+s)}, (S^{p^r})_\ell^{(s)}) \]
induced by $\sigma$ and the canonical map $\Ibold_\ell^{(r+s)}\to  (S^{p^s})_\ell^{(r)}$ is an isomorphism. This follows directly from Proposition \ref{pr-alg}(3). 
If $\mu\ne (p^s)$ then $S^\mu$ and $S^\mu\circ \Ibold^{(s)}_\ell$ are both tensor products of reduced functors. Hence the source and the target of $\Psi$ are zero by Lemma \ref{pirashvili}, so that $\Psi$ is an isomorphism.
\end{proof}

We now come to the main result of this subsection.

\begin{prop}\label{Psiiso}
For all $F$, the map $\Psi$ is an isomorphism.
\end{prop}

\begin{proof}
We are going to prove that $\Psi$ is an isomorphism by a spectral sequence argument. 

As a first step, we construct a lifting of our map $\Psi$ on the level of chain (bi)complexes. To be more specific,
we choose an injective coresolution $J^*$ of $F$, a projective resolution $P_*$ of $I^{(r)}\circ\Ibold^{(s)}_\ell$, and we consider the bicomplexes (the bicomplex $B^{m,n}$ has trivial vertical differentials)
\begin{align*}
&B^{m,n}=\Hom_\Pcal(I^{(r)},J^m)\otimes \Ext^n_{\Pbold}(\Ibold^{(s)}_\ell,\Ibold^{(s)}_\ell)^{(r)}\;,\\
&C^{m,n}= \Hom_\Pcal(I^{(r)},J^m)\otimes \Hom_{\Pbold}(P_n,\Ibold^{(r+s)}_\ell)\;,\\
&D^{m,n}=\Hom_{\Pbold}(P_n,J^m\circ\Ibold^{(r)}_\ell)\;.
\end{align*}
The homology of the total complexes of these bicomplexes are given by
\begin{align*}
&H^*(\Tot B)=\Ext^*_\Pcal(I^{(r)},F)\otimes \Ext^*_{\Pbold}(\Ibold^{(s)}_\ell,\Ibold^{(s)}_\ell)^{(r)}\;,\\
&H^*(\Tot C)=\Ext^*_\Pcal(I^{(r)},F)\otimes \Ext_{\Pbold}(\Ibold^{(r+s)}_\ell,\Ibold^{(r+s)}_\ell)\;,\\
&H^*(\Tot D)=\Ext^*_{\Pbold}(I^{(r)}\circ \Ibold^{(s)}_\ell,F\circ\Ibold^{(s)}_\ell)\;.
\end{align*}
We choose a basis of the graded space $\Ext^*_{\Pbold}(\Ibold^{(s)}_\ell,\Ibold^{(s)}_\ell)$ and for each element $b$ of this basis we choose a cocyle $z(b)$ representing $\sigma(b)$ in the complex $\Hom_{\Pbold}(P_*,\Ibold^{(r+s)}_\ell)$. This induces a morphism of bicomplexes
\[z:B^{m,n}\to C^{m,n}\] 
such that $H^*(\Tot z):H^*(\Tot B)\to H^*(\Tot C)$ is equal to $\id\otimes\sigma$. 
We also define a morphism of bicomplexes 
$\phi:C^{m,n}\to D^{m,n}$ as the composition
\[C^{m,n}\to \Hom_{\Pbold}(I^{(r)}\circ \Ibold_\ell^{(s)},J^m\circ \Ibold_\ell^{(s)})\otimes \Hom_{\Pbold}(P_n,\Ibold^{(r+s)}_\ell)\to D^{m,n}\]
where the first map is induced by precomposition by the functor $\Ibold_\ell^{(s)}$ and the second one by the composition of morphisms in $\Pbold$. Then the map $H^*(\Tot \phi)$ sends $x\otimes e^{(r)}$ to $(x\circ \Ibold^{(s)}_\ell)\cdot e^{(r)}$. Thus $H^*(\Tot (\phi\circ z))$ is equal to $\Psi$.

As a second step, we consider the first quadrant spectral sequences associated to the bicomplexes $B^{m,n}$ and $D^{m,n}$:
\begin{align*}
&E^{m,n}_1(B)=\Hom_\Pcal(I^{(r)},J^m)\otimes \Ext^n_{\Pbold}(\Ibold^{(s)}_\ell,\Ibold^{(s)}_\ell)^{(r)}\Rightarrow H^{m+n}(\Tot B)\;.
\\
&E^{m,n}_1(D)=\Ext^n_{\Pbold}(I^{(r)}\circ \Ibold^{(s)}_\ell,J^m\circ\Ibold^{(s)}_\ell)\Rightarrow H^{m+n}(\Tot D)\;.
\end{align*}
The morphism of bicomplexes $\phi\circ z$ induces a morphism of spectral sequences, which is equal to $\Psi$ on the abutment. Thus to prove that $\Psi$ is an isomorphism, it suffices to prove that $E_\infty^{m,n}(\phi\circ z)$ is an isomorphism.

By construction, $E^{m,*}_1(z\circ\phi)=\Psi$ on each column of index $m$, hence $E^{m,n}_1(z~\circ~\phi)$ is an isomorphism by Lemma \ref{lm-premier-cas}. Therefore $E^{m,n}_1(z\circ\phi)$ is an isomorphism on all pages, and we conclude that $E_\infty^{m,n}(\phi\circ z)$, hence $\Psi$, is an isomorphism.
\end{proof}

\subsection{The case $A=\Ibold^{(r)}_\ell$ and $B=\Ibold^{(r)}_{\overline{\ell}}$}\label{secsigmaprime}


We are going to slightly modify the construction made in Section 4.1. Recall from the beginning of the section that $\Ext_{\Pbold}^*(\Ibold^{(r+s)}_\ell,\Ibold^{(r+s)}_{\overline{\ell}})$ is a right $\Ext_{\Pbold}^*(\Ibold^{(r+s)}_\ell,\Ibold^{(r+s)}_\ell)$-module. In particular it inherits via $\sigma$ the structure of $\Ext_{\Pbold}^*(\Ibold^{(s)}_\ell,\Ibold^{(s)}_\ell)^{(r)}$-module. Then Proposition \ref{pr-alg}(2) implies directly the following:

\begin{prop}
There is an injective morphism of graded $\Ext_{\Pbold}^*(\Ibold^{(s)}_\ell,\Ibold^{(s)}_\ell)^{(r)}$-modules
\[\sigma':\Ext^*_{\Pbold}(\Ibold^{(s)}_\ell,\Ibold^{(s)}_{\overline{\ell}})^{(r)}\to \Ext_{\Pbold}^*(\Ibold^{(r+s)}_\ell,\Ibold^{(r+s)}_{\overline{\ell}})\]
given by $\sigma'(c_s^{(r)})=c_{s+r}$.
\end{prop}
%
%
Given $F \in \Pbold_{p^r}$, as in the previous section we define a graded linear map
\[\begin{array}{cccc}
\Psi':& \Ext^*_\Pcal(I^{(r)},F)\otimes \Ext^*_{\Pbold}(\Ibold^{(s)}_\ell,\Ibold^{(s)}_{\overline{\ell}})^{(r)}&\to &\Ext_{\Pbold}(I^{(r)}\circ \Ibold^{(s)}_\ell,F\circ \Ibold^{(s)}_{\overline{\ell}})\;,\\
& x\otimes e^{(r)} & \mapsto & (x\circ \Ibold^{(s)}_{\overline{\ell}})\cdot \sigma'(e^{(r)})
\end{array}
\]


\begin{prop}\label{Psiprimeiso}
For all $F$, the map $\Psi'$ is an isomorphism. 
\end{prop}

\begin{proof}
	The proof is conceptually similar to the one of Proposition \ref{Psiiso}. When $F=S^\mu$ is an injective cogenerator, the assertion follows in the same way by Proposition \ref{pr-alg}(3) if $\mu=(p^r)$ and from Lemma \ref{pirashvili} otherwise. 
	
	For an arbitrary $F$, take the very same $P_*$ and $J^*$ and form bicomplexes 
	\begin{align*}
	&B^{m,n}=\Hom_\Pcal(I^{(r)},J^m)\otimes \Ext^n_{\Pbold}(\Ibold^{(s)}_\ell,\Ibold^{(s)}_{\overline{\ell}})^{(r)}\;,\\
	&C^{m,n}= \Hom_\Pcal(I^{(r)},J^m)\otimes \Hom_{\Pbold}(P_n,\Ibold^{(r+s)}_{\overline{\ell}})\;,\\
	&D^{m,n}=\Hom_{\Pbold}(P_n,J^m\circ\Ibold^{(r)}_{\overline{\ell}})\;.
	\end{align*}
	Take a basis of $\Ext_{\Pbold}(\Ibold_{\ell}^{(s)}, \Ibold_{\overline{\ell}}^{(s)})$ and to any element $b$ of this basis associate a cocycle $z'(b)$ representing $\sigma'(b)$ in $\Hom_{\Pbold}(P_n,\Ibold^{(r+s)}_{\overline{\ell}})$. It induces a morphism of bicomplexes $z' : B^{m,n} \rightarrow C^{m,n}$ such that $H^*(\Tot z')=id \otimes \sigma'$. Then define a morphism $\varphi'$ as in the previous proof, replacing $\ell$ by $\overline{\ell}$ everywhere in the composition. The map $H^*(\Tot \varphi')$ sends $x \otimes e^{(r)}$ to $(x \circ \Ibold_{\overline{\ell}}^{(s)}) \cdot e^{(r)}$, so $H^*(\Tot(\varphi' \circ z'))$ is equal to $\Psi'$. In the second step, the spectral sequence associated to these new $B^{*,*}$ and $D^{*,*}$ are exactly the source and the target of $\Psi'$. The rest of the proof is a  word-for-word repetition.
\end{proof}

%

\subsection{The case $A=\PF \circ \Ibold_\ell^{(r)}$ and $B=\Ibold_\ell^{(r)}$}\label{sectau}

We use an easy lemma to get rid of the postcomposition by $\PF$:

\begin{lm}\label{extwithpf}
 For all superfunctors $F,G$ there is an odd isomorphism
   \[\Ext_{\Pbold}^*(\PF \circ F, G) \simeq \Ext_{\Pbold}^*(F, G)\] 
   given by Yoneda composition with the odd morphism $\pi_F \in \Hom_{\Pbold}(\PF \circ F, F)$ induced by parity change.
\end{lm}
\begin{proof}
Let $\pi : \Ibold \rightarrow \PF$ be the parity change morphism (namely, the odd natural transformation that acts as the identity on objects) and set $\pi_F := \pi \circ F : \PF \circ F \rightarrow F$. Precomposition by $\pi_F$ yields an odd isomorphism
\[\Hom_{\Pbold}(F, G) \simeq \Hom_{\Pbold}(\PF \circ F, G). \]
The isomorphism follows then easily at the $\Ext$ level.
\end{proof}

In our case $F=\Ibold_{\ell}^{(s)}$ we denote $\pi_F$ simply by $\pi_s$.
Since $\PF$ commutes with~$I^{(r)}$, $\Ext_{\Pbold}^*(I^{(r)} \circ \PF \circ \Ibold^{(s)}_\ell,\Ibold^{(r+s)}_\ell)$ is isomorphic to $\Ext_{\Pbold}^*(\PF \circ \Ibold_\ell^{(r+s)},\Ibold^{(r+s)}_\ell)$. For simplicity reasons, we will replace the first one by the second one in the two next subsections.


%
%
%

\begin{prop}\label{tauexistence}
	There is an injective morphism of graded super vector spaces
	\[\tau:\Ext_{\Pbold}^*(\PF \circ \Ibold^{(s)}_\ell,\Ibold^{(s)}_\ell)^{(r)}\to\Ext_{\Pbold}^*(\PF \circ \Ibold_\ell^{(r+s)},\Ibold^{(r+s)}_\ell) \]
	which fits into a commutative diagram
	\[	\begin{tikzcd}[row sep=large]
		\Ext_{\Pbold}^*(\PF \circ \Ibold^{(s)}_\ell,\Ibold^{(s)}_\ell)^{(r)} \arrow[r, "\tau"] & \Ext_{\Pbold}^*(\PF \circ \Ibold_\ell^{(r+s)},\Ibold^{(r+s)}_\ell) \\
		\Ext_{\Pbold}^*(\Ibold^{(s)}_\ell, \Ibold^{(s)}_\ell)^{(r)} \arrow[r, "\sigma"] \arrow[u, "(-)\cdot \pi_s^{(r)}"] & \Ext_{\Pbold}(\Ibold_{\ell}^{(r+s)}, \Ibold_{\ell}^{(r+s)}) \arrow[u, "(-)\cdot \pi_{r+s}"]
	\end{tikzcd} \]
\end{prop}

\begin{proof}
	Straightforward from the fact that the two vertical arrows are isomorphisms.
\end{proof}

As in Section \ref{secsigma}, $\tau$ induces a graded linear map
\[
\begin{array}{cccc}
	 \Omega: & \Ext^*_\Pcal(I^{(r)},F) \otimes \Ext^*(\PF \circ \Ibold^{(s)}_\ell,\Ibold^{(s)}_\ell)^{(r)}
	 &\to& \Ext^*(\PF \circ \Ibold_\ell^{(r+s)}, F \circ \Ibold^{(s)}_\ell) \\
 & x \otimes e^{(r)} & \mapsto & (x \circ \Ibold_{\ell}^{(s)}) \cdot \tau(e^{(r)})  
\end{array}
\]

\begin{prop}\label{omegaiso}
	For all $F$, the map $\Omega$ is an isomorphism.
\end{prop}

\begin{proof}
	The diagram of Proposition \ref{tauexistence} induces the following one
	\[
	\begin{tikzcd}[row sep=large]\label{appo}
	\Ext^*_\Pcal(I^{(r)},F) \otimes \Ext^*(\PF \circ \Ibold^{(s)}_\ell,\Ibold^{(s)}_\ell)^{(r)} \arrow[r, "\Omega"] & \Ext^*(\PF \circ \Ibold_\ell^{(r+s)},F \circ \Ibold^{(s)}_\ell) \\
	\Ext^*_\Pcal(I^{(r)},F) \otimes \Ext_{\Pbold}^*(\Ibold^{(s)}_\ell, \Ibold^{(s)}_\ell)^{(r)} \arrow[r, "\Psi"] \arrow[u, "id \otimes (-)\cdot \pi_s^{(r)}"] & \Ext_{\Pbold}(\Ibold_{\ell}^{(r+s)}, F \circ \Ibold_{\ell}^{(s)}) \arrow[u, "(-) \cdot \pi_{r+s}"]
	\end{tikzcd}
	\]
	which is then commutative as well. The two vertical arrows are obviously isomorphisms and $\Psi$ is an isomorphism by Proposition \ref{Psiiso}. The statement follows at once.
\end{proof}

\subsection{The case $A=\PF \circ \Ibold_\ell^{(r)}$ and $B= \Ibold_{\overline{\ell}}^{(r)}$}\label{sectauprime}

%


We use again Lemma \ref{extwithpf} with $F=\Ibold_{\ell}^{(s)}, \;  G=\Ibold_{\overline{\ell}}^{(s)}$ and $\pi_s$ the same parity change morphism.

\begin{prop}
	There is an injective morphism of graded super vector spaces
	\[\tau':\Ext_{\Pbold}^*(\PF \circ \Ibold^{(s)}_\ell,\Ibold^{(s)}_{\overline{\ell}})^{(r)}\to \Ext_{\Pbold}^*(\Ibold^{(r+s)}_\ell,\Ibold^{(r+s)}_{\overline{\ell}}) \]
	which fits into a commutative diagram
	\[	\begin{tikzcd}[row sep=large]
	\Ext_{\Pbold}^*(\PF \circ \Ibold^{(s)}_\ell,\Ibold^{(s)}_{\overline{\ell}})^{(r)} \arrow[r, "\tau'"] & \Ext_{\Pbold}^*(\PF \circ \Ibold_\ell^{(r+s)},\Ibold^{(r+s)}_{\overline{\ell}}) \\
	\Ext_{\Pbold}^*(\Ibold^{(s)}_\ell, \Ibold^{(s)}_{\overline{\ell}})^{(r)} \arrow[r, "\sigma'"] \arrow[u, "(-)\cdot \pi_r"] & \Ext_{\Pbold}(\Ibold_{\ell}^{(r+s)}, \Ibold_{\overline{\ell}}^{(r+s)}) \arrow[u, "(-)\cdot \pi_{r+s}"]
	\end{tikzcd} \]
\end{prop}
\begin{proof}
	Identical to the one of Proposition \ref{tauexistence}.
\end{proof}

Set the graded linear map
\[ \Omega' : \Ext^*_\Pcal(I^{(r)},F) \otimes \Ext_{\Pbold}^*(\PF \circ \Ibold^{(s)}_\ell,\Ibold^{(s)}_{\overline{\ell}})^{(r)}\to \Ext_{\Pbold}^*(I^{(r)} \circ \PF \circ \Ibold^{(s)}_\ell,F \circ \Ibold^{(s)}_{\overline{\ell}}) \]
induced by $\tau'$ in the same way as $\Omega$ is induced by $\tau$. The proof of the next proposition is identical to the proof of Proposition \ref{omegaiso}.

\begin{prop}\label{omegaprimeiso}
	For all $F$, the map $\Omega'$ is an isomorphism.
\end{prop}

\subsection{The general case}
We want to generalise the maps we have been constructing in the previous sections. Let $s$ be a positive integer, let $\mathcal{A}$ be full subcategory of $\Pbold_{p^s}$ consisting of additive superfunctors with finite dimensional values, and let $\mathcal{I}$ be the full subcategory of $\mathcal{A}$ generated by the set $\{ \Ibold^{(r)}_0, \Ibold^{(r)}_1, \PF \circ \Ibold^{(r)}_0\, \PF \circ \Ibold^{(r)}_1\}$. 
In the previous subsections, we have computed extensions of the form $\Ext^*_{\Pbold}(I^{(r)}\circ A,F\circ B)$ in a strict subset of cases where $A$ and $B$ are in $\mathcal{I}$. In this subsection, we generalise our computations to arbitrary $A,B \in \mathcal{A}$. We pass through the following standard fact:

\begin{lm}\label{technicallemma}
	Let $\mathcal{A}, \mathcal{B}$ be additive categories and $G,H : \mathcal{A} \rightarrow \mathcal{B}$ two additive functors. Let $\mathcal{I}$ be a full subcategory of $\mathcal{A}$ such that every object of $\mathcal{A}$ can be written as a finite direct sum of copies of objects in $Ob(\mathcal{I})$. Denote by $i: \mathcal{I} \hookrightarrow \mathcal{A}$ the canonical inclusion. Then:
	\begin{enumerate}
		\item Every natural transformation $T : G \circ i \rightarrow H \circ i$ extends uniquely to a natural transformation $\widetilde{T}: G \rightarrow H$.
		\item $\widetilde{T}$ is an isomorphism (monomorphism, epimorphism) if and only if $T$ is.
	\end{enumerate}
\end{lm}

It follows from Theorem \ref{classaddfunc} and Lemma \ref{rem} that the categories $\mathcal{A}$ and $\mathcal{I}$ defined at the beginning of this section satisfy the hypothesis of Lemma \ref{technicallemma}. We keep this notation for the proof of the following proposition.

\begin{prop}\label{eee}
Let $A$ and $B$ be two additive superfunctors of degree $p^s$. There is an injective morphism of graded super vector spaces
\[\sigma_{A,B}: \Ext^*_{\Pbold}(A,B)^{(r)}\to  \Ext^*_{\Pbold}(I^{(r)}\circ A,I^{(r)}\circ B)\]
natural with respect to $A$ and $B$, determined by the following conditions:
\begin{enumerate}
\item When $A=B=\Ibold_\ell^{(s)}$, $\sigma_{A,B}$ equals the map $\sigma$ of Section \ref{secsigma},
\item When $A=\Ibold_\ell^{(s)}$ and $B=\Ibold_{\overline{\ell}}^{(s)}$, $\sigma_{A,B}$ equals the map $\sigma'$ of Section \ref{secsigmaprime},
\item When $A=\PF \circ \Ibold_\ell^{(s)}$ and $B=\Ibold_\ell^{(s)}$, $\sigma_{A,B}$ equals the map $\tau$ of Section \ref{sectau},
\item When $A=\PF \circ \Ibold_\ell^{(s)}$ and $B=\Ibold_{\overline{\ell}}^{(s)}$, $\sigma_{A,B}$ equals the map $\tau'$ of Section \ref{sectauprime},
\item There is a commutative square
\[
\begin{tikzcd}[column sep =large]
\Ext^*_{\Pbold}(A,B)^{(r)}\ar{r}{\sigma_{A,B}} \ar{d}{\circ\PF}[swap]{\simeq}& \Ext^*_{\Pbold}(I^{(r)}\circ A,I^{(r)}\circ B)\ar{d}{\circ\PF}[swap]{\simeq}\\
\Ext^*_{\Pbold}(A\circ \PF,B\circ \PF)^{(r)}\ar{r}{\sigma_{A\circ \PF,B\circ \PF}}& \Ext^*_{\Pbold}(I^{(r)}\circ A\circ \PF,I^{(r)}\circ B\circ \PF).
\end{tikzcd}
\]
\end{enumerate}
\end{prop}

\begin{proof}
Whenever $A$ and $B$ are as in (1)-(4) set
\[ \sigma_{A \circ \PF, B \circ \PF}(e^{(r)}) := (\sigma_{A,B}(e^{(r)}\circ\PF)) \circ \PF\]
which gives by construction commutativity of the diagram in (5). Hence, $\sigma_{A,B}$ is completely determined by (1)-(5) when $A,B$ are in $\mathcal{I}$.
%
%
We want to conclude by invoking Lemma \ref{technicallemma} with $\mathcal{A}$ and $\mathcal{I}$ as above, $\mathcal{B}=\mathfrak{svec}$,  $G=\Ext_{\Pbold}^*(A,-)^{(r)}$ (resp. $\Ext_{\Pbold}^*(-,B)^{(r)}$) and $H=\Ext_{\Pbold}^*(I^{(r)} \circ A, I^{(r)} \circ -)$ (resp. $\Ext_{\Pbold}^*(I^{(r)} \circ -, I^{(r)} \circ B)$) for a fixed $A$ (resp. $B$) in $\mathcal{A}.$ The only thing to verify is that $G$ and $H$ are biadditive. For both this is automatic from biadditivity of $\Ext^*_{\Pbold}(-,-)$, additivity of the twist and additivity of the operation $I^{(r)} \circ -$.
\end{proof}

%
%
%

Define the graded linear map
\[ \begin{array}{cccc}
\Psi_{A,B}: & \Ext^*_\Pcal(I^{(r)},F)\otimes \Ext^*_{\Pbold}(A,B)^{(r)} & \to & \Ext_{\Pbold}^*(I^{(r)}\circ A,F\circ B) \\
& x \otimes e^{(r)} & \mapsto & (x\circ B)\cdot \sigma_{A,B}(e^{(r)})
\end{array}
\]

\begin{thm}\label{thm-main}
For all additive homogeneous superfunctors $A$ and $B$ of degree $p^s$ and for all homogeneous strict polynomial functor $F$ of degree $p^r$, $\Psi_{A,B}$ is an isomorphism.
\end{thm}

\begin{proof}
	When $A$ and $B$ are of the forms in (1)-(4) of the previous proposition, $\Psi_{A,B}$ is an isomorphism by (respectively) Propositions \ref{Psiiso}, \ref{Psiprimeiso}, \ref{omegaiso} and \ref{omegaprimeiso}. Then condition (5) of Proposition \ref{eee} implies that it is an isomorphism for all $A,B \in \mathcal{I}$. To conclude, we want to invoke Lemma \ref{technicallemma}(2) with the same category settings as the previous proof, $G=\Ext^*_\Pcal(I^{(r)},F)\otimes \Ext^*_{\Pbold}(A,-)^{(r)}$ (resp. $\Ext^*_\Pcal(I^{(r)},F)\otimes \Ext^*_{\Pbold}(-,B)^{(r)}$) and $H=\Ext_\Pcal^*(I^{(r)}\circ A,F\circ -)$ (resp. $\Ext_\Pcal^*(I^{(r)}\circ -,F\circ B)$) for a fixed $A$ (resp. $B$). The only non trivial property left to prove is the additivity of $H$ with respect to $B$. Let then $B, B'$ be additive. If $F=S^{p^r}_V:=S^{p^r}(V \otimes -)$ is an injective cogenerator
	\[ S^{p^r}_V \circ (B \oplus B') = S^{p^r}(V \otimes (B \oplus B')) \simeq S^{p^r}(V \otimes B \; \oplus \; V \otimes B') \simeq \bigoplus_{a+b=p^r} S^a_V \circ B \; \otimes \;  S^b_V \circ B'. \]
	Apply now $\Ext^*_{\Pbold}(I^{(r)} \circ A, -)$ to the last direct sum. Since $I^{(r)} \circ A$ is additive, by Pirashvili's vanishing lemma (\ref{pirashvili}) the only terms which survive are the ones corresponding to $a=0$ and $b=0$. That is,
	\[\Ext_{\Pbold}^*(I^{(r)} \circ A, S^{p^r}_V \circ (B \oplus B')) \simeq \Ext_{\Pbold}^*(I^{(r)} \circ A, S^{p^r}_V \circ B) \; \oplus \; \Ext_{\Pbold}^*(I^{(r)} \circ A, S^{p^r}_V \circ B')\]  
	which proves the case $F=S^{p^r}_V$. For an arbitrary $F$, take an injective coresolution $F \hookrightarrow J^*$ and form the spectral sequences
	\begin{align*}
	E_1^{s,t} = \Ext_{\Pbold}^t(I^{(r)} \circ A, J^s \circ B \; \oplus J^s \circ B') \Rightarrow \Ext_{\Pbold}^{s+t}(I^{(r)} \circ A, F \circ B \; \oplus \; F \circ B') \\
	F_1^{s,t} = \Ext_{\Pbold}^t(I^{(r)} \circ A, J^s \circ (B \oplus B')) \Rightarrow  \Ext_{\Pbold}^{s+t}(I^{(r)} \circ A, F \circ (B  \oplus B'))
	\end{align*}
	The inclusions $B, B' \subset B \oplus B'$ induce a morphism
	\[ \psi : \Ext_{\Pbold}^{s+t}(I^{(r)} \circ A, F \circ B \: \oplus \: F \circ B') \rightarrow \Ext_{\Pbold}^{s+t}(I^{(r)} \circ A, F \circ (B \oplus B')) \]
	as well as a morphism of spectral sequences $\varphi : E^{*,*} \rightarrow F^{*,*}$ which equals $\psi$ on the abutment. But we know that $\varphi$ is an isomorphism since the $J^s$ are injective. Then $\psi$ is an isomorphism and the theorem is proved.
\end{proof}

\section{A general conjecture}

We now rewrite the computations of Section \ref{sec-ext-comp} in order to place them in a more general context. 

Firstly, as we have already recalled it in Section \ref{sec-31}, the evaluation of a strict polynomial functor on a $\mathbb{Z}$-graded vector space $E^*$ of finite total dimension yields a $\mathbb{Z}$-graded vector space $F(E^*)$. Replacing $E^*$ by $E^*\otimes V$ for a finite dimensional vector space $V$ considered as homogeneous of degree $0$, we obtain a graded vector space $F(E^*\otimes V)$ which is actually a strict polynomial functor of the variable $V$, denoted by $F_{E^*}$ \cite[Lemma 2.8]{TouzeENS}. 
We need to generalise this to the case of arbitrary $\mathbb{Z}$-graded vector spaces. 
\begin{definition}\label{parametr}
Let $E^*$ be a $\mathbb{Z}$-graded vector space, and let $F \in \Pcal_d$. 
\begin{enumerate}
\item
We let $F_{E^*}$ be the $\mathbb{Z}$-graded $d$-homogeneous strict polynomial functor defined by 
\[F_{E^*}=\mathop{colim}_{{E'}^*\subset E^*}F_{{E'}^*}\;,\]
where the colimit is taken over the poset of finite dimensional graded vector subspaces ${E'}^*$ of $E^*$, ordered by inclusion. We call $F_{E^*}$ the \emph{lower parametrisation of $F$ by $E^*$}.
\item We also denote by $F^{E^*}$ the functor $F_{\Hom_k(E^*,k)}$ and we call it the \emph{upper parametrisation of $F$ by $E^*$}.
\end{enumerate}
\end{definition}

Write $\Pcal_d^*$ for the category of $\mathbb{Z}$-graded $d$-homogeneous strict polynomial functors. Then $F_{E^*}, F^{E^*} \in \Pcal_d^*$ and they have the following properties: 

\begin{prop}\label{parametr-properties}
\begin{enumerate}
\item Lower and upper parametrisations by $E^*$ define exact functors
\[\Pcal_d\to \Pcal_d^*\;.\]
Moreover, both operations are natural with respect to $E^*$.
\item If $E$ has finite total dimension, then $F_{E^*}(V)=F(E^*\otimes V)$ is the graded vector space described in Section \ref{sec-31}.
\item Assume that $E$ is degreewise finite-dimensional and zero in negative degrees. For all degrees $i$, let $E^{*\le i}$ be the graded subspace of $E^*$ which is equal to $E^*$ in degrees less or equal to $i$ and zero in degrees higher than $i$. Then the inclusion $E^{*\le i}\hookrightarrow E^*$ induces a monomorphism $F_{E^{*\le i}} \to F_{E^{*}}$ which is an isomorphism in degrees less or equal to $i$.
\item Make the same hypothesis on $E^*$ as in point (3). Then there are bigraded isomorphisms, natural in $F$, $G$ and $E^*$:
\[\Ext^*_\Pcal(F^{E^*},G)\simeq \Ext^*_\Pcal(F,G_{E^*})\;.\]
\end{enumerate}
\end{prop}
\begin{proof}
(1) follows from \cite[Lemma 2.8]{TouzeENS} and exactness of filtered colimits. 

If $E^*$ has finite total dimension, the poset in Definition \ref{parametr}(1) is finite and contains $E^*$ as a final object. This implies that the colimit is equal to $F_{E^*}$ as stated in (2).

To prove (3), the graded monomorphism $\tau_i : F_{E^{* \le i}} \hookrightarrow F_{E^*}$ is simply the one induced by functoriality by the canonical inclusion $E^{* \le i} \hookrightarrow E^*$. Moreover, the hypothesis implies that each $E^{* \le i}$ has finite total dimension. Since any finite-dimensional subspace of $E^*$ is included in $E^{* \le n}$ for some $n$, $F_{E^*}$ is equal to the colimit of the chain of inclusions $F_{E^{* \le 0}} \subset F_{E^{* \le 1}} \subset \dots \subset F_{E^{* \le n}} \subset \dots$, i.e. to the quotient of $\bigoplus_{d \ge 0} F_{E^{* \le d}}$ obtained by identifying all components of the same degree. In particular, the components of degree less than $i$ have all representants in $F_{E^{* \le i}}$. Then the source and the target of $\tau_i$ have the same dimension in degrees less than $i$, making $\tau_i$ an isomorphism in that case.
 
Let us prove (4). By point (3) it is sufficient to provide bigraded isomorphisms 
\[\Ext^*_\Pcal(F^{E^{* \le i}},G)\simeq \Ext^*_\Pcal(F,G_{E_{* \le i}}) \]
for all $i\ge 0$. In other words, we may suppose that $E^*$ is of finite total dimension and use the explicit description of point (2). Denote by $V^\vee = \Hom(V,k)$ the dual of a (possibly graded) vector space. One has
 \[
 \begin{array}{cc}
 (\Gamma^{d,V})^{E^*}(W) = \Gamma^{d,V}((E^*)^\vee \otimes W) = \Gamma^d \Hom(V, (E^*)^\vee \otimes W) \\
 \simeq \Gamma^d \Hom(E^* \otimes V, W) =  \Gamma^{d, E^* \otimes V}(W)
 \end{array}
 \]
giving an isomorphism $(\Gamma^{d,V})^{E^*} \simeq \Gamma^{d, E^* \otimes V}$ natural in $V$ and $E^*$. In particular, in each degree, $(\Gamma^{d,V})^{E^*}$ is a projective object in $\Pcal_d$. A double application of Yoneda lemma \ref{yoneda} gives then for each $G$ 
\[ \Hom_\Pcal((\Gamma^{d,V})^{E^*}, G) \simeq G(E^* \otimes V) \simeq \Hom_\Pcal(\Gamma^{d,V}, G_{E^*})
\]
which respects the gradings by \cite[Cor. 2.12]{FS}. By projectivity of $\Gamma^{d,V}$ and $(\Gamma^{d,V})^{E^*}$, the isomorphism on the $\Ext^*$ follows at once for $F=\Gamma^{d,V}$. For a general $F$, take a projective resolution $P_* \rightarrow F$. By exactness of the parametrisation and by the fact that each $P_n$ is a sum of $\Gamma^{d,V}$, $(P_*)^{E^*}$ is a projective coresolution of $F^{E^*}$ degreewise. Then from the previous case we deduce 
\[ \Ext^*_\Pcal(F^{E^*}, G) \simeq H^*(\Hom_\Pcal((P_*)^{E^*}, G)) \simeq H^*(\Hom_\Pcal(P_*, G_{E^*})) \simeq \Ext^*_\Pcal(F, G_{E^*}) \]

\end{proof}

We can now give a reformulation of Theorem \ref{thm-main}.
\begin{cor}
For all additive homogeneous superfunctors $A$ and $B$ of degree $p^s \; \; (s > 0)$ with finite dimensional values and for all homogeneous strict polynomial functor $F$ of degree $p^r$, there is a graded isomorphism (where we take the total degree on the left-hand side), natural in $A$, $B$, $F$:
\[
\Ext^*_{\Pcal}(I^{(r)},F_{\Ext^*_{\Pbold}(A,B)})\simeq 
\Ext^*_{\Pbold}(I^{(r)}\circ A, F\circ B)\;.
\]
\end{cor}
\begin{proof}
Set $E^* := \Ext^*_{\Pbold}(A,B)^{(r)}$. By hypothesis $A$ and $B$ have finite dimensional values, so they can be written as a finite sum of copies of Frobenius twists. From Proposition \ref{pr-alg} it follows that $E^*$ is finite-dimensional degreewise. Then, by additivity of $I^{(r)}$ and of $\Ext^*$ in the first variable, the isomorphism of Theorem \ref{thm-main} can be rewritten as 
\[ \Ext^*_{\Pcal}((I^{(r)})^{E^*},F)\simeq E^* \otimes \Ext^*_{\Pcal}(I^{(r)},F) \simeq 
\Ext^*_{\Pbold}(I^{(r)}\circ A, F\circ B)
\]
natural in $F,A,B$. An application of Proposition \ref{parametr-properties}(4) provides the desired isomorphism.
\end{proof}

This result is very similar to the computation of extensions between twisted strict polynomial functor in the classical setting, see \cite[Thm 4.6]{TouzeSurvey} and \cite{Chalupnik, TouzeUnivSS}. This leads us to the following conjecture.
\begin{conj}\label{conj}
For all additive homogeneous superfunctors $A$ and $B$ of degree $p^s\; \; (s > 0)$ with finite dimensional values and for all homogeneous strict polynomial functor $F$ and $G$ of degree $p^r$, there is a graded isomorphism (where we take the total degree on the left-hand side), natural in $A$, $B$, $F$, $G$:
\[
\Ext^*_{\Pcal}(G,F_{\Ext^*_{\Pbold}(A,B)})\simeq 
\Ext^*_{\Pbold}(G\circ A, F\circ B)\;.
\]
\end{conj}


\bibliographystyle{plain}
\bibliography{bibliographie.bib}

\end{document}